\documentclass[11pt]{article}
\usepackage{latexsym,amsmath}

\def\fullpage {
\addtolength{\topmargin}{-2 cm}
\addtolength{\oddsidemargin}{-1.5cm} \addtolength{\textwidth}{+3 cm}
\addtolength{\textheight}{+4 cm}}
\fullpage

\usepackage{amsthm}
\usepackage{amssymb}
\usepackage{epsfig}
\usepackage{amsmath}
\usepackage{graphicx}

\newcommand{\beq}[1]{\begin{equation}\label{eq:#1}}
\newcommand{\eeq}{\end{equation}}

\begin{document}
\newtheorem{theorem}{Theorem}
\newtheorem{corollary}{Corollary}
\newtheorem{lemma}{Lemma}
\newtheorem{proposition}{Proposition}
\newtheorem{conjecture}{Conjecture}
\newcommand\eps{\varepsilon}

\newcommand{\brac}[1]{\left(#1\right)}
\newcommand{\bfrac}[2]{\brac{\frac{#1}{#2}}}
\newcommand{\rdown}[1]{\left\lfloor#1\right\rfloor}
\newcommand{\rdup}[1]{\left\lceil#1\right\rfloor}
\newcommand{\me}{\mathrm{e}}
\newcommand{\ee}{\epsilon}
\newcommand{\ex}{\mathrm{ex}}
\newcommand{\Bin}{\mathrm{Bin}}
\newcommand{\cG}{{{\cal G}}}
\newcommand{\cK}{{\cal K}}
\newcommand{\cH}{{\cal H}}
\newcommand{\cP}{{\cal P}}
\newcommand{\cJ}{{\cal J}}
\newcommand{\cE}{{\cal E}}
\def\Q{\mathcal{Q}}
\newtheorem{definition}[theorem]{Definition}
\def\FF{\mathcal{F}}
\def\F{\mathcal{F}}
\def\eps{\varepsilon}
\newcommand{\eee}{{\mathbb E}}
\def\HH{\mathcal{H}}
\parindent=0pt

\title{Almost all cancellative triple systems are tripartite}

\author{J\'ozsef Balogh\thanks{{Department of
Mathematics, U.C. California at San Diego, 9500 Gilmann Drive, La
Jolla, Department of Mathematics; and  University of Illinois, 1409
W. Green Street, Urbana, IL 61801, USA}; e-mail:
{jobal@math.uiuc.edu;} research  supported in part   by NSF CAREER
Grant DMS-0745185 and DMS-0600303, UIUC Campus Research Board Grants
09072 and 08086, and OTKA Grant K76099.}  \quad  and \quad Dhruv
Mubayi\thanks{Department of Mathematics, Statistics, and Computer
Science, University of Illinois at Chicago, IL 60607;  email:
mubayi@math.uic.edu; research  supported in part by  NSF grant DMS
0653946.}}

\maketitle

\vspace{-0.4in}

\begin{abstract}
A triple system is cancellative if  no three of its distinct edges satisfy $A \cup B=A \cup C$.
 It is tripartite if it has a vertex partition into three parts such that every edge has exactly one point in each part.
   It is easy to see that every tripartite triple system is cancellative. We prove that almost all cancellative triple systems with
    vertex set $[n]$ are tripartite. This sharpens a  theorem of Nagle and R\"odl \cite{NR} on the number of cancellative triple
     systems.  It also extends recent work of Person and Schacht \cite{PSch} who proved a similar result
      for triple systems without the Fano configuration.

Our proof  uses the  hypergraph regularity lemma of Frankl and R\"odl \cite{FR}, and a stability
theorem for cancellative triple systems  due to Keevash and the second author \cite{KM}.
\end{abstract}

\section{Introduction}
 Let  $F$ be a fixed graph or hypergraph. Say that a (hyper)graph is $F$-free if it contains
  no copy of $F$ as a (not necessarily induced) sub(hyper)graph. Beginning with a result of
Erd\H os-Kleitman-Rothschild \cite{EKR}, there has been much work
concerning the number and structure of $F$-free graphs with vertex
set $[n]$ (see, e.g. \cite{EFR, KPR, PS1, BBS1, BBS2, BBS3, BSAM}). The
strongest of these results essentially state that for a large class of graphs $F$, most
of the $F$-free graphs with vertex set $[n]$ have a similar
structure to the $F$-free graph with the maximum number of edges.
Many of these results use  the Szemer\'edi regularity lemma.

With the development of the hypergraph regularity Lemma, these problems can be attacked for hypergraphs.
 For brevity, we refer to a $3$-uniform hypergraph as a triple system or $3$-graph.

 {\bf Definition.}
{\em For a $3$-graph $F$ let  $Forb(n, F)$ denote the set of  (labeled) $F$-free
$3$-graphs on vertex set $[n]$.  }

 The first result in this direction was due to Nagle and R\"odl \cite{NR} who proved that for a fixed 3-graph $F$,  
$$|Forb(n, F)| \le 2^{{\rm ex}(n, F) + o(n^3)},$$
where ex$(n,F)$ is the maximum number of edges in an $F$-free triple system on $n$ vertices.
 Since there is no extremal result for hypergraphs similar to Tur\'an's theorem for graphs,
  one cannot expect a general result that characterizes the structure of almost all $F$-free triple systems for large classes of $F$.
   Nevertheless, much is known about the extremal numbers for a few specific 3-graphs $F$ and one could hope
   to obtain characterizations for these $F$. Recently, Person and Schacht \cite{PSch} proved the first result of this kind, by showing that almost all
   triple systems on $[n]$ not containing a Fano configuration are $2$-colorable.
   The key property that they used was the linearity of the Fano plane,
   namely the fact that every two edges of the Fano plane share at most one vertex.
   This enabled them to apply the (weak) $3$-graph regularity lemma, which is almost
   identical to Szemer\'edi's regularity lemma.  They then proved an embedding lemma for linear
    hypergraphs essentially following ideas from Kohayakawa-Nagle-R\"odl-Schacht \cite{KNRS}.

It is well-known that such an embedding lemma fails to hold for
non-linear $3$-graphs unless one uses the (strong) $3$-graph
regularity lemma, and operating in this environment is more
complicated. In this paper, we address the situation for a
particular non-linear $F$ using this approach.

A triple system is tripartite or $3$-partite if it has a vertex
partition into three parts such that every  edge has exactly one
point in each part. Denote by $T(n)$  the number of $3$-partite
$3$-graphs on $[n]$.  Let
$$s(n):=
\left\lfloor\frac{n}{3}\right
\rfloor \left\lfloor\frac{n+1}{3}\right\rfloor
\left\lfloor\frac{n+2}{3}\right\rfloor\sim \frac{n^3}{27}$$
be the  maximum number
  of edges in a $3$-partite triple system with $n$ vertices. A triple system is cancellative if  $A \cup B = A \cup C$ implies that $B=C$ for edges $A,B,C$.
 Every tripartite triple system is cancellative and Katona conjectured, and Bollob\'as \cite{bollobas:74}
  proved that the maximum number of edges in a cancellative triple system with $n$ vertices is $s(n)$.
It is easy to see that a cancellative triple system is one that contains no copy of
$$F_5=\{123, 124, 345\} \quad \hbox{ and } \quad
K_4^-=\{123, 124, 234\}.$$ Later Frankl and F\"uredi \cite{frankl+furedi:83} sharpened Bollob\'as'
 theorem by proving that ex$(n, F_5)=s(n)$ for $n>3000$ (this was improved to $n>33$ in \cite{KM}).

Our main result is the following.

\begin{theorem}\label{mainf}
Almost all $F_5$-free $3$-graphs on $[n]$ are $3$-partite. More
precisely there is a constant $C$ such that
\begin{equation}\label{induction}
|Forb(n,F_5)|< \left(1+ 2^{Cn-\frac{2n^2}{45}}\right) T(n).\end{equation}
\end{theorem}

Theorem \ref{mainf}  clearly implies the same result for
cancellative $3$-graphs which is stated in the abstract.
As
mentioned before, the proof of Theorem \ref{mainf} uses the strong
hypergraph regularity lemma, and stability theorems.

Using the fact that
$$\frac{4\cdot 3^n}{n^2} 2^{s(n)}<T(n)<3^n2^{s(n)}$$ (see Lemma \ref{tnincreasing}), we get the following improvement over the general result  of
 Nagle and R\"odl~\cite{NR} which only implies that
$|Forb(n,F_5)|<2^{s(n)+o(n^3)}$.

\begin{corollary} As $n \rightarrow \infty$,
$$\log_2|Forb(n,F_5)|=s(n)+n\log_2 3 +\Theta(\log n).$$
\end{corollary}
 In a
forthcoming paper \cite{t5}, we shall characterize the structure of
almost all $F$-free $3$-graphs, where $F=\{123,124,125,345\}$.  Note
that such a fine statement as Theorem~\ref{mainf} is rare even for
graphs: Pr\"omel and Steger \cite{PS1} characterized the structure
of almost all $F$-free graphs when $F$ has a color-critical edge,
and Balogh, Bollob\'as and Simonovits~\cite {BBS3} when
$F=K(2,2,2)$.

\section{Stability}
The key idea in the proof of Theorem \ref{mainf} is to reduce the
problem to $3$-graphs that are almost $3$-partite. We  associate a hypergraph with its edge set.

 For a triple system $\HH$
 with a $3$-partition $P$ of its vertices, say that an edge is {\em crossing} if it has exactly one point in each part,
 otherwise say that it is {\em non-crossing}.  Let $D_P$ be the set of non-crossing edges.
 An {\it
optimal partition} $X \cup Y \cup Z$ of a triple system $\HH$ is a $3$-partition of the
vertices of $\HH$ which minimizes the number of non-crossing edges.
Let $D=D_{\HH}$ be the number of non-crossing ({\it bad}) edges in an
optimal partition $X \cup Y \cup Z$.  Define
$$Forb(n, F_5, \eta):=\{\HH \subset [n]^3: F_5 \not\subset \HH\hbox{ and } D_{\HH}\le \eta n^3\}.$$

The first part of the proof of Theorem \ref{mainf} is the following
result, which we will prove in Section~\ref{mostrip}.

\begin{theorem} \label{stablef}
For every $\eta>0$, there exists $\nu>0$ and $n_0$ such that
 if $n>n_0$, then $$|Forb(n,F_5)-Forb(n, F_5, \eta)|<2^{(1-\nu)\frac{n^3}{27}}.$$
\end{theorem}

\section{Hypergraph Regularity}\label{hypreg}

In this section, we quickly define the notions required to state the
hypergraph regularity Lemma. These concepts will be used in
Section~\ref{mostrip} to prove Theorem \ref{stablef}. Further
details can be found in \cite{FR} or \cite{NR}.  As mentioned before
we associate a hypergraph with its edge set.

A $k$-{\it partite cylinder} is a $k$-partite graph $G$ with
$k$-partition $V_1, \ldots, V_k$, and we write $G=\cup_{i<j}
G^{ij}$, where $G^{ij}=G[V_i \cup V_j]$ is the bipartite subgraph
 of $G$ with parts $V_i$ and $V_j$.
 For $B \in [k]^3$, the $3$-partite cylinder $G(B)=\cup_{\{i,j\} \in [B]^2} G^{ij}$ is called a {\it triad}.
  For a $2$-partite cylinder $G$,
 the {\it density} of the pair $V_1, V_2$ with respect to $G$ is $d_G(V_1, V_2)=\frac{|G^{12}|}{|V_1||V_2|}$.

Given  an integer $l>0$ and real $\epsilon>0$, a $k$-partite
cylinder $G$ is called an $(l, \epsilon, k)$-{\it cylinder} if for
every $i<j$, $G^{ij}$ is $\epsilon$-regular with density $1/l$.
 For a $k$-partite cylinder $G$, let $\cK_3(G)$ denote the
 $3$-graph on $V(G)$ whose edges correspond to triangles of $G$. An easy
consequence of these definitions is the following fact.

\begin{lemma} {\bf (Triangle Counting Lemma)}  \label{tlemma} For integer $l>0$ and  real $\theta>0$,
 there exists $\epsilon>0$ such that every $(l,\epsilon,3)$-cylinder $G$ with $|V_i|=m$ for all $i$ satisfies
$$|\cK_3(G)| =(1\pm \theta)\frac{m^3}{l^3}.$$
\end{lemma}

We now move on to $3$-graph definitions. A $k$-{\it partite
$3$-cylinder} is a $k$-partite 3-graph $\cH$ with $k$-partition
$V_1, \ldots, V_k$. Here $k$-partite means that every edge of $\cH$
has at most one point in each $V_i$.  Often we will say that these
edges are crossing, and the edges that have at least two points is
some $V_i$ are non-crossing. Given $B \in [k]^3$, let
$\cH(B)=\cH[\cup_{i \in B}V_i]$. Given $k$-partite cylinder $G$ and
$k$-partite $3$-cylinder $\cH$ with the same vertex partition, say
that $G$ {\it underlies} $\cH$ if $\cH \subset \cK_3(G)$.  In other
words, $\cH$ consists  only of triangles in $G$. Define the density
$d_{\cH}(G(B))$ of $\cH$ with respect to the triad $G(B)$ as the
proportion of edges of $\cH$ on top of triangles of $G(B)$, if the
latter quantity is positive, and zero otherwise. This definition
leads to the more complicated definition of $\cH$ being $(\delta,
r)$-regular with respect to  $G(B)$, where $r>0$ is an integer
and $\delta>0$. If in addition $d_{\cH}(G(B))=\alpha \pm \delta$, then
say that $\cH$ is $(\alpha, \delta, r)$-{\it regular} with respect
to $G(B)$. We will not give the precise definitions of $(\alpha,
\delta, r)$-regularity, and it suffices to take this definition as a
``black box" that will be used later.

For a vertex set $V$, an $(l, t, \gamma, \epsilon)$-partition
 $\cP$ of $[V]^2$ is a partition $V=V_0 \cup V_1 \cup \cdots \cup V_t$
together with a collection of edge disjoint bipartite graphs $P_{a}^{ij}$,
 where $1\le i<j\le t, 0\le a\le l_{ij} \le l$ that satisfy the following properties:

(i) $|V_0|<t$ and $|V_i|=\lfloor \frac{n}{t} \rfloor:=m$ for each
$i>0$,

(ii) $\cup_{a=0}^{l_{ij}}P_{a}^{ij}=K(V_i, V_j)$ for all
$1\le i<j\le t$, where $K(V_i, V_j)$ is the complete bipartite graph
with parts $V_i, V_j$,

(iii) all but $\gamma{t \choose 2}$ pairs $\{v_i, v_j\}$, $v_i \in V_i, v_j \in V_j$,
are edges of $\epsilon$-regular bipartite graphs $P_{a}^{ij}$, and

(iv) for all but $\gamma{t \choose 2}$ pairs $\{i,j\} \in [t]^2$,
we have $|P_0^{ij}|\le \gamma m^2$ and $d_{ P_{a}^{ij} }(V_i, V_j)=(1\pm \epsilon)\frac{1}{l}$
for all $a \in [l_{ij}]$.

Finally, suppose that $\cH \subset [n]^3$ is a $3$-graph and $\cP$
is an $(l, t, \gamma, \epsilon)$-partition of $[n]^2$ with
$m_{\cP}=|V_1|$. For each triad $P \in \cP$, let
$\mu_P=\frac{|\cK_3(P)|}{m_{\cP}^3}$. Then $\cP$ is $(\delta,
r)$-regular if
$$\sum\{\mu_P: \hbox{$P$ is a $(\delta, r)$-irregular triad of $\cP$}\} <\delta\left(\frac{n}{m_{\cP}}\right)^3.$$

We can now state the Regularity Lemma due to Frankl and R\"odl
\cite{FR}.

\begin{theorem} {\bf (Regularity Lemma)} \label{rl}
For every $\delta, \gamma$ with $0<\gamma\le 2\delta^4$, for all
integers $t_0, l_0$ and for all integer-valued functions $r=r(t, l)$
and all functions $\epsilon(l)$, there exist $T_0, L_0, N_0$ such
that every $3$-graph $\cH \subset [n]^3$ with $n\ge N_0$ admits a
$(\delta, r(t,l))$-regular $(l, t, \gamma, \epsilon(l))$-partition
for some $t,l$ satisfying $t_0 \le t<T_0$ and $l_0\le l<L_0$.
\end{theorem}

To apply the Regularity Lemma above, we need to define a cluster
hypergraph and state an accompanying embedding Lemma,  sometimes
called the Key Lemma.  Given a $3$-graph $\cJ$, let $\cJ^2$ be the
set of pairs that lie in an edge of $\cJ$.

{\bf Cluster $3$-graph.} For given constants $k, \delta, l, r,
\epsilon$ and sets  $\{\alpha_B: B \in [k]^3\}$ of nonnegative
reals, let $\cH$ be a $k$-partite 3-cylinder with parts $V_1,
\ldots, V_k$, each of size $m$.  Let $G$ be a graph, and $\cJ
\subset [k]^3$ be a $3$-graph such that the following conditions are
satisfied.

(i) $G=\cup_{\{i,j\} \in \cJ^2} G^{ij}$ is an underlying cylinder of
$\cH$ such that for all $\{i,j\} \in \cJ^2$, $G^{ij}$ is an $(l,
\epsilon, 2)$-cylinder.

(ii) For each $B \in \cJ$, $\cH(B)$ is $(\alpha_B, \delta, r)$-regular with respect to the triad $G(B)$.

Then we say that $\cJ$ is the {\it cluster $3$-graph} of $\cH$.

\begin{lemma} {\bf (Embedding Lemma)} \label{elemma} Let $k \ge 4$ be fixed.
 For all  $\alpha>0$, there exists $\delta>0$ such that for $l>\frac{1}{\delta}$, there exists
  $r, \epsilon$ such that the following holds: Suppose that $\cJ$ is the cluster $3$-graph
  of $\cH$ with underlying cylinder $G$ and parameters $k, \delta, l, r, \epsilon, \{\alpha_B: B \in [k]^3\}$
   where $\alpha_B \ge \alpha$ for all $B \in \cJ$.  Then $\cJ \subset \cH$.
\end{lemma}

For a proof of the Embedding Lemma, see \cite{NR}.

\section{Most $F_5$-free triple systems are almost tripartite}\label{mostrip}

In this section we will prove Theorem \ref{stablef}. We will need
the following stability result proved in \cite{KM}. The constants
have been adjusted for later use.

\begin{theorem} {\bf (Keevash-Mubayi \cite{KM})} \label{km}
For every $\nu''>0$, there exist $\nu', t_2$ such that every
$F_5$-free $3$-graph on $t>t_2$ vertices and at least
$(1-2\nu')\frac{t^3}{27}$ edges has a $3$-partition for which the
number of non-crossing edges is at most $\nu'' t^3$.
\end{theorem}

Given $\eta>0$, our constants will obey the following hierarchy:
$$\eta\gg \nu''\gg \nu' \gg \nu \gg \sigma, \theta
\gg  \alpha_0, \frac{1}{t_0} \gg \delta \gg \gamma >\frac{1}{l_0} \gg\frac{1}{r}, \epsilon \gg \frac{1}{n_0}.$$
Before proceeding with further details regarding our constants,
we define the {\it binary entropy function} $H(x):=
-x\log_2 x- (1-x)\log_2 (1-x).$
We use the fact  that for $0<x< 0.5$ we
have $$\binom{n}{xn}<2^{H(x)n}.$$ Additionally, if $x$ is sufficiently small
then
\begin{equation} \label{x} \sum_{i=0}^{xn} \binom{n}{i}<2^{H(x)n}.\end{equation}

{\bf Detailed definition of constants.}

Set
\begin{equation} \label{nu''def}\nu''=\frac{\eta}{1000}\end{equation}
 and suppose that $\nu'_1$ and $t_2$ are the outputs of Theorem \ref{km} with input $\nu''$.  Put
\begin{equation} \label{nu'}
\nu'=\min\{\nu'_1, \nu''\} \quad \hbox{ and } \quad \nu=(\nu')^4.\end{equation}
We choose
  \begin{equation} \label{theta}
  \theta=\frac{\nu}{4(1-\nu)}.\end{equation}

Choose $\sigma_1$ small enough  so that
\begin{equation} \label{sigma}
\left(1-\frac{\nu}{2}\right)\frac{n^3}{27}+o(n^3)+H(\sigma)n^3\le \left(1-\frac{\nu}{3}\right)\frac{n^3}{27}\end{equation}
holds for sufficiently large $n$.  In fact the function denoted by
$o(n^3)$  will actually be seen to be of order $O(n^2)$ so
(\ref{sigma}) will hold for sufficiently large $n$. Choose $\sigma_2$ small enough so that (\ref{x}) holds for $\sigma_2$. Let
$$\sigma=\min\{\sigma_1, \sigma_2\}.$$
Next we consider the Triangle Counting Lemma (Lemma \ref{tlemma}) which provides an
$\epsilon$ for each  $\theta$ and $l$. Since $\theta$ is fixed, we may let
$\epsilon_1=\epsilon_1(l)$ be the output of Lemma \ref{tlemma} for each integer $l$.

For $\sigma$ defined above, set
\begin{equation} \label{alpha}\delta_1=\alpha_0=\frac{\sigma}{100} \quad
\hbox{ and } \quad t_1=\left\lceil \frac{1}{\delta_1} \right\rceil.\end{equation}
Let
$$t_0=\max\{t_1, t_2, 33\}.$$
Now consider the Embedding Lemma (Lemma \ref{elemma}) with inputs $k=5$ and $\alpha_0$
defined above.  The Embedding Lemma gives $\delta_2=\delta_2(\alpha_0)$, and
we set
\begin{equation} \label{delta} \delta=\min\{\delta_1, \delta_2\}, \quad
\quad \gamma=\delta^4, \quad \quad l_0=\frac{2}{\delta}.\end{equation}
For each integer $l>\frac{1}{\delta}$, let $r=r(l)$ and
$\epsilon_2=\epsilon_2(l)$ be the outputs of Lemma~\ref{elemma}. Set
\begin{equation} \label{epsilonl}\epsilon=\epsilon(l)=\min\{\epsilon_1(l), \epsilon_2(l)\}.\end{equation}

With these constants, the Regularity Lemma (Theorem \ref{rl}) outputs $N_0$.  We choose
$n_0$ such that $n_0>N_0$  and every $n>n_0$ satisfies
(\ref{sigma}).

\medskip

{\bf Proof of the Theorem \ref{stablef}.}

We will prove that $$|Forb(n,F_5)-Forb(n, F_5, \eta)|<2^{(1-\frac{
\nu}{3})\frac{n^3}{27}}.$$ This is of course equivalent to Theorem
\ref{stablef}.

For each $\cG \in Forb(n,F_5)-Forb(n, F_5, \eta)$, we use the
Hypergraph Regularity Lemma, Theorem~\ref{rl}, to obtain a $(\delta,
r)$-regular $(l, t, \gamma, \epsilon)$-partition $\cP=\cP_{\cG}$.
The input constants for Theorem~\ref{rl} are as defined above and
then Theorem \ref{rl} guarantees constants $T_0, L_0, N_0$ so that
every $3$-graph $\cG$ on $n>N_0$ vertices admits a $(\delta,
r)$-regular $(l, t, \gamma, \epsilon)$-partition $\cP$ where $t_0
\le t \le T_0$ and $l_0\le l \le L_0$. To this partition $\cP$,
associate a {\em density vector} $s=(s_{\{i,j,k\}_{a,b,c}})$ where $1
\le i<j<k\le t$ and $1\le a,b,c \le l$ and
$$d_{\cG}(P_a^{ij} \cup P_b^{jk} \cup P_c^{ik})\in [s_{\{i,j,k\}_{a,b,c}}\delta, (s_{\{i,j,k\}_{a,b,c}}+1)\delta].$$

For each $\cG \in
Forb(n, F_5, \eta)$, choose one $(\delta, r)$-regular $(l, t,
\gamma, \epsilon)$-partition $\cP_{\cG}$ guaranteed by
Theorem~\ref{rl}, and let $\cP=\{\cP_1, \ldots, \cP_p\}$ be the set
of all such partitions over the family $Forb(n, F_5, \eta)$.
Define
an equivalence relation on $Forb(n, F_5, \eta)$ by letting $\cG\sim
\cG'$ iff

1) $\cP_{\cG}=\cP_{\cG'}$ and

2) $\cG$ and $\cG'$ have the same density vector.

The number of equivalence classes $q$ is the number of partitions
times the number of  density vectors.  Consequently,
$$q\le \left({T_0 +1\choose 2}(L_0+1)\right)^{n \choose 2}
\left(\frac{1}{\delta}\right)^{{T_0+1 \choose 2}(L_0+1)^3}<2^{O(n^2)}.$$

We will show that each equivalence class $C(\cP_i, s)$ satisfies
\begin{equation} \label{C} |C(\cP_i, s)|=2^{(1-\frac{\nu}{2})\frac{n^3}{27}+H(\sigma)n^3}.
\end{equation}
Combined with the upper bound for $q$ and (\ref{sigma}), we obtain
$$|Forb(n, F_5, \eta)|\le 2^{O(n^2)}2^{(1-\frac{\nu}{2})\frac{n^3}{27}+H(\sigma)n^3}\le
2^{(1-\frac{\nu}{3})\frac{n^3}{27}}.$$

For the rest of the proof, we fix an equivalence class $C=C(\cP, s)$ and we will show the upper bound in (\ref{C}).
 We may assume that $\cP$ has vertex partition
  $[n]=V_0\cup V_1\cup \cdots \cup V_t$, $|V_i|=m=\lfloor \frac{n}{t}\rfloor$ for all $i\ge 1$,
  and system of bipartite graphs $P_{a}^{ij}$, where $1\le i<j\le t, 0\le a\le l_{ij} \le l$.

 Fix $\cG \in C$.  Let $\cE_0\subset \cG$ be the set of triples that either

 (i) intersect $V_0$, or

 (ii) have at least two points in some
$V_i, i\ge 1$, or

 (iii) contain a pair in $P_0^{ij}$ for some $i,j$, or

 (iv) contain a pair in some $P_{a}^{ij}$ that is not $\epsilon$-regular with density $\frac1l$.

 Then
 $$|\cE_0|\le tn^2+t\left(\frac{n}{t}\right)^2 n +\gamma{t \choose 2}n+2\gamma{t \choose 2}\left(\frac{n}{t}\right)^2 n.$$

  Let $\cE_1 \subset \cG-\cE_0$ be the set of triples $\{v_i, v_j, v_k\}$ such that either

  (i) the three bipartite graphs of $\cP$  associated with the pairs within the triple form
  a triad $P$ that is not $(\delta, r)$-regular with respect to $\cG(\{i,j,k\})$, or

  (ii) the density $d_{\cG}(P)<\alpha_0$.

 Then
 $$|\cE_1|\le 2\delta t^3\left(\frac{n}{t}\right)^3(1+\theta) +\alpha_0
 {t \choose 3}l^3\left(\frac{n}{t}\right)^3 \frac{1}{l^3}.$$

  Let $\cE_{\cG}=\cE_0\cup \cE_1$.
  Now (\ref{alpha}) and (\ref{delta}) imply that
$$|\cE_{\cG}|\le \sigma n^3.$$
Set $\cG'=\cG-\cE_{\cG}$.

Next we define $\cJ^C=\cJ^C(\cG)\subset [t]^3 \times [l] \times [l] \times [l]$ as follows:
 For $1 \le i <j <k \le t,\  1 \le a,b,c \le l$,  we have $\{i,j,k\}_{a,b,c} \in \cJ^C$ if and only if

(i) $P=P_a^{ij} \cup P_b^{jk} \cup P_c^{ik}$ is an $(l, \epsilon,
3)$-cylinder, and

(ii) $\cG'(\{i,j,k\})$ is $(\overline{\alpha}, \delta, r)$-regular with respect to $P$, where
 $\overline{\alpha}\ge \alpha_0$.

We view $\cJ^C$ as a multiset of triples on $[t]$. For each
$\phi:{[t]\choose 2} \rightarrow [l]$, let $\cJ_{\phi}\subset \cJ^C$ be the
$3$-graph on $[t]$ corresponding to the function $\phi$ (without
parallel edges).  In other words, $\{i,j,k\} \in \cJ_{\phi}$ iff the triples of $\cG$
that lie on top of the triangles of $P_{a}^{ij} \cup P_{b}^{jk} \cup
P_{c}^{ik}\ $, $a=\phi(ij),\  b=\phi(jk),\ c=\phi(ik)$,
 are $(\overline{\alpha}, \delta, r)$-regular and the underlying bipartite graphs $P_{a}^{ij}, P_{b}^{jk}, P_{c}^{ik}$
 are all $\epsilon$-regular with density $1/l$.

By our choice of the constants in (\ref{delta}) and (\ref{epsilonl}),
 we see that any ${\cal F}\subset \cJ_{\phi}$ with five vertices is a cluster $3$-graph for $\cG$, and hence by the Embedding Lemma ${\cal F} \subset \cG$.
 Since $F_5 \not\subset \cG$, we conclude that $F_5 \not\subset \cJ_{\phi}$.
  It was shown in \cite{KM} that for $t \ge 33$, we have ex$(t, F_5)\le \frac{t^3}{27}$.
  Since we know  that $t \ge 33$, we conclude that
  $$|\cJ_{\phi}|\le \hbox{ex} (t, F_5) \le \frac{t^3}{27}$$ for each $\phi :{[t]\choose 2} \rightarrow [l]$.  Recall from (\ref{nu'}) that $\nu'=\nu^{1/4}$.

\begin{lemma}\label{markov}
 Suppose that $|\cJ^C|>(1-\nu)\frac{l^3t^3}{27}$. Then for at least $(1-\nu') l^{{t \choose 2}}$ of the
functions $\phi:{[t]\choose 2} \rightarrow [l]$ we have
$$|\cJ_{\phi}| \ge (1-\nu')\frac{|\cJ^C|}{l^3}.$$
\end{lemma}

\begin{proof}   Form the following bipartite graph: the vertex partition is
$\Phi \cup \cJ^C$ , where
$$\Phi=\left\{\phi: {[t]\choose 2} \rightarrow [l]\right\}$$
and the edges are of the form $\{\phi, \{i,j,k\}_{abc}\}$ if and only if $\phi \in \Phi$,
 $\{i,j,k\}_{abc}\in \cJ^C$ where $\phi(\{i,j\})=a,\  \phi(\{j,k\})=b,\  \phi(\{i,k\})=c$.
 Let $E$ denote the number of edges in this bipartite graph. Since each $\{i,j,k\}_{abc} \in \cJ^C$ has degree precisely
$l^{{t \choose 2}-3}$, we have
$$E=|\cJ^C| l^{{t \choose 2}-3}.$$
Note that the degree of $\phi$ is $|\cJ_{\phi}|$.
 Suppose for contradiction that the number of $\phi$ for which
 $|\cJ_{\phi}| \ge (1-\nu')\frac{|\cJ^C|}{l^3}$ is less than $(1-\nu') l^{{t \choose 2}}$.
 Then since $|\cJ_{\xi}|\le \frac{t^3}{27}$ for each $\xi\in \Phi$,  we obtain the upper bound
$$E\le (1-\nu') l^{{t \choose 2}}\frac{t^3}{27} + \nu'l^{{t \choose 2}}(1-\nu')\frac{|\cJ^C|}{l^3}.$$
Dividing by $l^{{t \choose 2}-3}$ then yields
$$|\cJ^C| \le (1-\nu')l^3\frac{t^3}{27} + \nu'(1-\nu')|\cJ^C|.$$
Simplifying, we obtain
$$(1-\nu'(1-\nu'))|\cJ^C|\le (1-\nu')l^3\frac{t^3}{27}.$$
  The lower bound $|\cJ^C|>(1-\nu)\frac{l^3t^3}{27}$ then gives
$$(1-\nu'(1-\nu'))(1-\nu)< 1-\nu'.$$
Since $\nu'=\nu^{1/4}$, the left hand side expands to
$$1-\nu'+\nu^{1/2}-\nu+\nu^{5/4}-\nu^{3/2}>1-\nu'.$$
This contradiction completes the proof.\end{proof}

{\bf Claim 1.}
$$|\cJ^C|\le (1-\nu)\frac{l^3 t^3}{27}.$$
Once we have proved Claim 1, the proof is complete by following the argument which is very
similar to that in \cite{NR}. Define
$$S^C=\bigcup_{ \{i,j,k\}_{abc} \in \cJ^C} \cK_3(P_a^{ij} \cup P_b^{jk} \cup P_c^{ik}).$$
The Triangle Counting Lemma implies that $|\cK_3(P_a^{ij} \cup P_b^{jk} \cup P_c^{ik})| <\frac{m^3}{l^3}(1+\theta)$.
 Now  Claim 1 and $(\ref{theta})$  give
$$|S^C| \le \frac{m}{l^3}(1+\theta)|\cJ^C|\le m^3(1+\theta)(1-\nu)\frac{t^3}{27}<m^3\frac{t^3}{27}
\left(1-\frac{\nu}{2}\right)\le \frac{n^3}{27}\left(1-\frac{\nu}{2}\right).$$
Since $\cG' \in S^C$ for every $\cG \in C$,
$$|\{\cG': \cG \in C\}|\le 2^{(1-\frac{\nu}{2})\frac{n^3}{27}}.$$
Each $\cG \in C$ can be written as $\cG=\cG' \cup \cE_{\cG}$. In view of (\ref{x}) and
$|\cE_{\cG}|\le \sigma n^3$, the number of $\cE_{\cG}$ with $\cG \in
C$ is at most $\sum_{i\le \sigma n^3} {n^3 \choose i}\le
2^{H(\sigma)n^3}$.
 Consequently,
$$|C| \le 2^{(1-\frac{\nu}{2})\frac{n^3}{27}+H(\sigma)n^3}
$$
and we are done.

{\bf Proof of Claim 1.} Suppose to the contrary that $|\cJ^C|>
(1-\nu)\frac{l^3 t^3}{27}$.
 We  apply  Lemma~\ref{markov}  and conclude that for most functions $\phi$ the corresponding triple
  system $\cJ_{\phi}$ satisfies
$$|\cJ_{\phi}| \ge (1-\nu')\frac{|\cJ^C|}{l^3} > (1-\nu')(1-\nu)\frac{t^3}{27}>(1-2\nu')\frac{t^3}{27}.$$
By  Theorem \ref{km}, we conclude that for all of these $\phi$, the
triple system $\cJ_{\phi}$ has a $3$-partition where the number of
non-crossing edges is at most  $\nu'' t^3$.  We also conclude that
the number of crossing triples that are not edges of $\cJ_{\phi}$ is
at most
\begin{equation} \label{nu''} \left(\frac{2\nu'}{27}+\nu''\right)t^3<\frac{5}{3}\nu''t^3. \end{equation}

Fix one such $\phi$ and let the optimal partition of $\cJ_{\phi}$  be $P_{\phi}=X \cup Y \cup Z$.
 Let $P=V_X \cup V_Y \cup V_Z$ be the corresponding vertex partition of $[n]$.
 In other words, $V_X$ consists of the union of all those parts $V_i$ for which $i \in X$ etc.
   We will show that $P$ is a partition of $[n]$ where the number of
   non-crossing edges $|D_P|$ is fewer than $\eta n^3$.  This contradicts the fact that
    $\cG \in Forb(n,F_5)-Forb(n, F_5, \eta)$ and completes the proof of Theorem \ref{stablef}.

We have  argued earlier  that $|\cE_{\cG}|\le \sigma n^3 \le
\frac{\eta}{2}n^3$ so it suffices to prove that $|D_P -\cE_{\cG}|\le
\frac{\eta}{2}n^3$.

Call a $\xi: {[t]\choose 2} \rightarrow [l]$ {\it good} if it satisfies the
conclusion of Lemma~\ref{markov}, otherwise call it {\it bad}. For
each $\xi$ and edge $\{i,j,k\} \in \cJ_{\xi}$, we have $a,b,c$
defined by $a=\xi(\{i,j\})$ etc. let $\cG_{\xi}$ be the union, over
all $\{i,j,k\} \in \cJ_{\xi}$, of the edges of $\cG$ that lie on top
of the triangles in $P_{a}^{ij} \cup P_{b}^{jk} \cup P_{c}^{ik}$.
Let $D_{\xi}$ be the set of edges in $\cG_{\xi}$ that are
non-crossing with respect to $P=V_X \cup V_Y \cup V_Z$. We will
estimate $|D_P-\cE_{\cG}|$ by summing  $|D_{\xi}|$ over all $\xi$.
Please note that each $e\in D_P-\cE_{\cG}$  lies in exactly  $l^{{t
\choose 2}-3}$ different $D_{\xi}$ due to the definition of $\cJ^C$.
Summing over all $\xi$ gives
$$l^{{t \choose 2}-3} |D_P-\cE_{\cG}| =\sum_{\xi:{[t]\choose 2} \rightarrow [l]}|D_{\xi}|\le \sum_{\xi\  good} |D_{\xi}|
+\sum_{\xi\  bad} |D_{\xi}|.$$ Note that for a given edge $\{i,j,k\}
\in \cJ_{\phi}$ the number of edges in $\cG_{\phi}$ corresponding to
this edge is the number of edges in $V_i \cup V_j \cup V_k$ on top
of triangles formed by the three bipartite graphs, each of which is
$\epsilon$-regular of density $1/l$.  By the Triangle Counting
Lemma, the total number of such triangles is at most
$$2|V_i||V_j||V_k|\left(\frac1l\right)^3<2\left(\frac{n}{t}\right)^3 \left(\frac1l\right)^3.$$
By Lemma~\ref{markov}, the number of bad $\xi$ is at most $\nu'
l^{{t\choose 2}}$. So we have
$$\sum_{\xi\  bad} |D_{\xi}|\le \nu' l^{{t\choose 2}}{t \choose 3}2
\left(\frac{n}{t}\right)^3 \left(\frac1l\right)^3<\nu'l^{{t\choose
2}-3}n^3.$$ It remains to estimate $\sum_{\xi\  good} |D_{\xi}|$.

 Fix a good $\xi$ and let the optimal partition of $\cJ_{\xi}$ be $P_{\xi}=A \cup B \cup C$
 (recall that we know the number of non-crossing edges with respect to to this partition is less than $\nu''t^3$).

{\bf Claim 2.} The number of crossing edges of $P_{\xi}$ that are
non-crossing edges of $P_{\phi}$ is at most $100\nu''t^3$.

Suppose that  Claim 2 was true. Then we would obtain
$$\sum_{\xi \ good} |D_{\xi}| \le l^{{t\choose 2}}\left[100\nu''t^3(\frac{n}{t})^3\frac{2}{l^3}+
\nu''t^3(\frac{n}{t})^3\frac{2}{l^3}\right]\le l^{{t\choose
2}-3}\left[202\nu''n^3\right].$$ Explanation: We consider the
contribution from the non-crossing edges of $P_{\phi}$ that are (i)
crossing edges of $P_{\xi}$ and (ii) non-crossing edges of
$P_{\xi}$. We do not need to consider the contribution from the
crossing edges of $P_{\phi}$ since by definition, these do not give
rise to edges of $D_P$.

Altogether, using (\ref{nu''def}) we  obtain
$$|D_P-\cE_{\cG}| \le (202\nu''+\nu')n^3<\frac{\eta}{2} n^3$$
and the proof is complete.  We now  prove Claim 2.

{\bf Proof of Claim 2.}  Suppose for contradiction that the number
of crossing edges of $P_{\xi}$ that are non-crossing edges of
$P_{\phi}$ is more than $100\nu'' t^3$. Each of these edges intersects at most $3{t \choose 2}$ other edges of $\cJ_{\xi}$,
 so by the greedy algorithm we can find a
collection of at least $50\nu''t$ of these edges that form a matching $M$. Pick one such edge $e=\{k,k', k''\}\in M$ and assume that $k$ and $k'$ lie in same part $U$ of $P_{\phi}$. Let $d$ be the number of ways to choose a set of two triples $\{f, f'\}$ with $f=\{i,j,k\}, f'=\{i,j,k'\}$, $i,j \not\in U\cup \{k''\}$ and $i$ and $j$ lie in distinct parts of $P_{\phi}$.
Since $|\cJ_{\phi}|>(1-2\nu')\frac{t^3}{27}$, $|D_{P_{\phi}}|\le \nu'' t^3$ and $\nu', \nu''$ are sufficiently small
$$d\ge (\min\{|X|, |Y|, |Z|\} -1)^2\ge \frac{t^2}{10}.$$ As $\{e,f,f'\}\cong F_5$
there
are at least $d$ potential copies of $F_5$ that we can
form using $e$ and two crossing triples $f,f'$ of $P_{\phi}$.  Suppose that $f=\{i,j,k\}, f'=\{i,j,k'\}$
are both in  $\cJ_{\phi}$ for one such choice of $\{f,f'\}$. Consider the following eight bipartite graphs:
$$G^{ij}=P_{\phi(\{i,j\})}^{ij}, \quad G^{jk}=P_{\phi(\{j,k\})}^{jk} \quad G^{ik}=
P_{\phi(\{i,k\})}^{ik}\quad G^{jk'}=P_{\phi(\{j,k'\})}^{jk'}\quad
G^{ik'}=P_{\phi(\{i,k'\})}^{ik'}$$
$$G^{kk'}=P_{\xi(\{k, k'\})}^{kk'}\quad G^{k'k''}=P_{\xi(\{k', k''\})}^{k'k''}\quad G^{kk''}=P_{\xi(\{k, k''\})}^{kk''}.$$
Set $G=\bigcup G^{uv}$ where the union is over the eight bipartite graphs defined above.
 Since $\{e,f,f'\} \subset \cJ_{\phi} \cup \cJ_{\xi}$, the 3-graph  $J=\{e, f,f'\}$ associated with $G$ and $\cG$
 is a cluster 3-graph.
By (\ref{delta}) and (\ref{epsilonl}), we may apply the Embedding Lemma and obtain the contradiction  $F_5 \subset \cG$.
We conclude that $f''\not\in \cJ_{\phi}$ for some $f''\in \{f, f'\}$.

To each $e \in M$ we have associated at least $d$ triples  $f''\not\in \cJ_{\phi}$.
Since $M$ is a matching and $|e \cap f''| = 1$, each such $f''$ is counted at most three times.
Summing over all $e\in M$, we obtain at least $\frac{|M|d}{3}\ge \frac{5}{3}\nu'' t^3$ triples $f''$ that are crossing with
 respect to $P_{\phi}$ but are not edges of $\cJ_{\phi}$.  This contradicts (\ref{nu''}) and completes the proof. \qed

\section{Proof of Theorem \ref{mainf}}\label{proofmain}
In this section we complete the proof of Theorem \ref{mainf}. We
begin with some preliminaries.

 \subsection{Inequalities}
We shall use Chernoff's inequality as follows:

\begin{theorem}\label{chernoff}
Let $X_1,\ldots,X_m$ be independent $\{0,1\}$ random variables with
$P(X_i=1)=p$ for each $i$. Let $X=\sum_i X_i$. Then the following
inequality holds for
$a>0$:\\
$$P(X < \eee X - a) < \exp(-a^2/(2pm)).$$
\end{theorem}

We will use the following easy statement.

\begin{lemma}\label{matching}
Every graph $G$ with $n$ vertices contains a matching of size at
least $\frac{|G|}{2n}$.
\end{lemma}

\begin{proof}
Assume that there is a maximal matching of size $r$. The $2r$
vertices of the matching can cover at most $2rn$ edges, by the
maximality of the matching there is no other edge in $G$.
\end{proof}

Recall that $T(n)$ is the number of $3$-partite $3$-graphs with vertex set $[n]$ and $s(n)=\lfloor\frac{n+2}{3}\rfloor\cdot \lfloor\frac{n+1}{3}\rfloor\cdot
\lfloor\frac{n}{3}\rfloor.$ For a 3-partition $A, B, C$ of a 3-graph, and $u \in A, v \in B$, write $L_C(u,v)$ or simply $L(u,v)$ for the set of $w \in C$ such that $uvw$ is an edge.
 As usual, the multinomial coefficient ${n \choose a,b,c}=\frac{n!}{a!b!c!}$.

\begin{lemma}\label{tnincreasing}As $n \rightarrow \infty$ we have
\begin{equation} \label{T(n)}
\left(\frac{1}{6}-o(1)\right) \ \binom{n}{\lfloor\frac{n+2}{3}\rfloor,
\lfloor\frac{n+1}{3}\rfloor, \lfloor\frac{n}{3}\rfloor} 2^{s(n)}\  < T(n) \  <  \
3^n 2^{s(n)}.\end{equation} In addition, \begin{equation}\label{tnn-2} T(n-2)<\left(n^{2} 2^{-\frac{2n^2}{9}+n}\right) T(n).\end{equation}
\end{lemma}

\begin{proof}
 For the upper bound in (\ref{T(n)}), observe that $3^n$ counts the number of $3$-partitions
of the vertices, and the exponent is the maximum number of
crossing edges that a $3$-partite 3-graph can have.

 For the lower
bound we count the number of (unordered) $3$-partitions where this
equality can be achieved. Each such $3$-partition gives rise to $2^{s(n)}$ $3$-partite $3$-graphs.
The number of such 3-partitions of $[n]$ is at least
$$\frac{1}{6} \ \binom{n}{\lfloor\frac{n+2}{3}\rfloor,
\lfloor\frac{n+1}{3}\rfloor, \lfloor\frac{n}{3}\rfloor}.$$
We argue next that most of the  3-partite 3-graphs obtained in this way are different.  More precisely, we show below that for any given 3-partition $P$ as above, most 3-partite 3-graphs with 3-partition $P$ have a unique 3-partition (which must be $P$).
Given a $3$-partition $U_1,U_2,U_3$ of $[n]$, if the crossing edges are added randomly, then Chernoff's inequality gives that almost all 3-graphs generated
satisfy the following two conditions:

(i) for all $ u\in U_i, v\in U_j$, where $\{i,j,\ell\}=\{1,2,3\}$  we have $|L_{U_\ell}(u,v)| > n/10$

(ii) for $\{i,j,\ell\}=\{1,2,3\}$ and  for every $ A_i \subset U_i, A_j\subset U_j$ with $|A_i|, |A_j|> n/10$ and $v\in U_\ell$,  the number of crossing edges
 intersecting each of $A_i,A_j$ and containing $v$ is at least $|A_1||A_2|/10$.

 If $\HH$ has $3$-partition $U_1,U_2,U_3$ of $[n]$, and it satisfies conditions (i) and (ii), then the $3$-partition is unique. Indeed, take $u,v$ lying in an edge, then $u$, $v$ and $L(u,v)$ are in different parts, where
$|L(u,v)|> n/10$, so for $w\in L(u,v)$, $L(u,w)$ is in the same part as $v$ and $L(v,w)$ is in the same part as  $u$.  Now by (ii) the rest of the vertices must lie in a unique part.

 To prove \eqref{tnn-2}  first note that if $a+b+c=n$, then ${n \choose a,b,c}$ is maximized for $a=\lfloor (n+2)/3\rfloor, b=\lfloor (n+1)/3\rfloor, c=\lfloor n/3\rfloor$. This implies that
 $$3^n=\sum_{a+b+c=n} {n \choose a,b,c} \le {n+2 \choose 2} \binom{n}{\lfloor\frac{n+2}{3}\rfloor,
\lfloor\frac{n+1}{3}\rfloor, \lfloor\frac{n}{3}\rfloor}<(0.6)n^2
\binom{n}{\lfloor\frac{n+2}{3}\rfloor,
\lfloor\frac{n+1}{3}\rfloor, \lfloor\frac{n}{3}\rfloor}
.$$
 Together with \eqref{T(n)} we obtain
 $$
\frac{T(n-2)}{T(n)} < \frac{3^{n-2} 2^{s(n-2)}}{(\frac{1}{6}-o(1)) \binom{n}{\lfloor\frac{n+2}{3}\rfloor,
\lfloor\frac{n+1}{3}\rfloor, \lfloor\frac{n}{3}\rfloor} 2^{s(n)}} < n^2 2^{s(n-2)-s(n)}.$$
It is easy to see that $s(n)-s(n-2)\ge 2n^2/9-n$, and the result follows.
\end{proof}

\subsection{Lower Density}

\begin{definition}\label{defdensity}
A vertex partition $U_1,U_2,U_3$ of a 3-graph $\F$ is {\it $\mu$-lower dense}
if each of the following conditions
 are satisfied:\\
(i) For every $i$ if $A_i\subset U_i$ with $|A_i|\ge \mu n$ then
$$|\{E\in\F: \ |E\cap A_i|=1, \text{ for } 1\le i\le 3\}\ >\ |A_1| \cdot |A_2| \cdot |A_3| \cdot 2^{-3}.$$
(ii) Let $\{i,j,\ell\}=\{1,2,3\}$, $A_i\subset U_i$ with $|A_i|\ge
\mu n$, $G\subset U_j\times U_\ell$ with
 $|G|\ge \mu^2 n^2.$
Then
$$|\{E\in\F: \ |E\cap A_i|=1,\ E-A_i\in G\}|\ >\ |A_i|\cdot |G| \cdot 2^{-3}.$$
(iii) Let $\{i,j,\ell\}=\{1,2,3\}$, $A_i\subset U_i$ and $A_j\subset
U_j$ with $|A_i|,|A_j|\ge \mu n$, and
 $G$ be a matching on
 $U_\ell$ with $|G|\ge \mu n.$
Set
$$\F_{A_i, A_j, G}=\{ \{C, D\} \in \F^2: C-U_\ell=D-U_\ell,\
 |C\cap A_i|=|C\cap A_j|=1,\{(C\cap U_\ell),( D\cap U_\ell)\}\in G\}.$$
Then $$|\F_{A_i, A_j, G}| \ge \frac{|A_i|\cdot |A_j|\cdot  |G|}{
2^{7}}.$$ (iv) For every $i$ we have $||U_i|-n/3|<\mu n$.
\end{definition}

For $\mu>0$ let $Forb(n,F_5,\eta,\mu)\subset Forb(n,F_5,\eta)$ be the
family of $\mu$-lower dense hypergraphs.

\begin{lemma}\label{density}
For every $\eta $ if $\mu^3\ge 10^3H(6\eta)$ then  for $n$ large enough
$$|Forb(n,F_5,\eta)-Forb(n,F_5,\eta,\mu)|\ < \ 2^{n^3(1/27-\mu^3/40)}.$$
\end{lemma}

\begin{proof}
We wish to count the number of  $\HH \in Forb(n, F_5, \eta)-Forb(n,
F_5, \eta, \mu)$. The number of ways to choose a $3$-partition of
$\HH$ is at most $3^n$. Given a particular $3$-partition $P=(U_1,
U_2, U_3)$, the number of ways the at most $\eta n^3$ bad edges
could be placed is at most
$$\sum_{i\le \eta n^3} {{n \choose 3} \choose i}< 2^{H(6\eta){n \choose 3}}.$$
If $|U_i-n/3|>\mu n$ for some $i$, then the number of possible
crossing edges is at most
$$n^3(1/27-\mu^2/4+\mu^3/4)<n^3(1/27-\mu^2/5).$$ We conclude that the number of
$\HH \in Forb(n, F_5, \eta)-Forb(n, F_5, \eta, \mu)$ for which there exists a partition that
 fails property (iv) is at most
$$ f(n, \eta) 2^{n^3(1/27-\mu^2/5)},$$
where
$$f(n, \eta)=3^n \cdot 2^{H(6\eta){n \choose 3}}.$$
Since $\HH \not\in Forb(n, F_5, \eta, \mu)$ it fails to satisfy one of the four conditions in Definition \ref{defdensity}.
  For a fixed partition $P$ and choice of bad edges, we may view $\HH$ as a probability space where we choose
   each crossing edge with respect to $P$ independently with probability $1/2$.
 The total number of ways to choose the crossing edges is at most $2^{n^3/27}$ (an upper bound on the size of the
 probability space) so we obtain that $|Forb(n,F_5,\eta)-Forb(n,F_5,\eta,\mu)|$ is upper bounded by
 $$f(n, \eta) \cdot 2^{n^3/27} \cdot Prob(\HH \hbox{ fails (i) or (ii) or (iii)})+f(n, \eta) 2^{n^3(1/27-\mu^2/5)}.$$

 We will consider each of these probabilities separately and then use the union bound.  First however, note that
 the number of
choices for $A_i \subset U_i$ as in Definition \ref{defdensity} is at most $2^n$ and the number of ways $G$
could be chosen is at most $2^{n^2}$.

(i) Since  $|A_1||A_2||A_3| \ge \mu^3 n^3$, Chernoff's inequality
gives
$$Prob(\HH \hbox{ fails (i)}) \le 2^{3n} \cdot
\exp(-\mu^3 n^3/16).$$
(ii) Since  $|A_i||G| \ge \mu^3 n^3$, Chernoff's inequality gives
$$Prob(\HH \hbox{ fails (ii)}) \le 2^{n} \cdot 2^{n^2}\cdot \exp(-\mu^3 n^3/16).$$
(iii) Since $|A_i||A_j||G| \ge \mu^3 n^3$ and  both edges $C$ and $D$ must be present, we apply
Chernoff's inequality with $m=|A_i||A_j||G|/2$ and $p=1/4$. The number of matchings $G$ is at most $(n^2)^{n/2}=2^{n\log_2 n}$, so
$$Prob(\HH \hbox{ fails (iii)}) \le 2^{2n} \cdot 2^{n\log_2 n}\cdot \exp(-\mu^3 n^3/32).$$
The lemma now follows since $10^3H(6\eta)\le \mu^3,$ and $n$ is sufficiently large.
\end{proof}

\subsection{There is no  bad vertex}

Let $\HH \in Forb(n,F_5,\eta,\mu)$, assume $n$ is large enough, and
$U_1,U_2,U_3$ is an optimal
 partition of $\HH$, with $x\in U_1$.
For a vertex $y$ let $L_{i,j}(y)$ denote the set of edges of $\HH$
containing $y$, and additionally intersecting $U_i$ and $U_j$. In
particular, $L_{i,i}(y)$ is the set of edges of $\HH$ which contain
$y$, and their other vertices are in $U_i$.

The aim of this subsection is to prove the following lemma, which shows that the number of bad edges
containing a vertex is small.

\begin{lemma}\label{lowdegree}
  Each of the followings is satisfied for $x\in U_1$.\\
(i) $|L_{1,1}(x)|\ <\  2\mu n^2.$\\
(ii) $|L_{1,2}(x)|\ < \ 2\mu n^2.$\\
(iii) $|L_{2,2}(x)|\ < \ 2\mu n^2.$\\
(iv) $|L_{1,3}(x)|\ < \ 2\mu n^2.$\\
(v) $|L_{3,3}(x)|\ < \ 2\mu n^2.$
\end{lemma}

\begin{proof}
(i) If $|L_{1,1}(x)|>2\mu n^2$  then by Lemma~\ref{matching} $
\{E-x:\ E\in L_{1,1}(x)\}$  contains a matching $G$ with size at
least $\mu n$. Then using Definition~\ref{defdensity} (iii) (with
$G, A_i=U_2, A_j=U_3$) for $\HH$, we find $y,z\in U_1, a\in U_2,
b\in U_3$ such that $xyz,yab,zab\in \HH$,
 yielding an $F_5\subset \HH$, a contradiction.

(ii) Suppose for contradiction that $|L_{1,2}(x)|\ge  \ 2\mu n^2$. By the optimality of the partition
 $|L_{1,2}(x)|\le
|L_{2,3}(x)|$, otherwise $x$ could be moved to $U_3$ to decrease the
number of bad edges.
 We shall  use  property (ii) in  Definition~\ref{defdensity}. We use it
with $G=\{E-x:\ E\in L_{1,2}(x)\}$ and
$$A_3=\{z\in U_3: \ \exists \hbox{ crossing edges } E_1,E_2\in \HH
\text{ with } \{x,z\}\subset E_1\cap E_2\}.$$
  Note that $|A_3|\ge
\mu n$ as $|L_{2,3}(x)|\ge 2\mu n^2$. Since $\HH$ is $\mu$-lower
dense, we find $abz \in \HH$ with $xab \in L_{1,2}(x)$ and $z \in
A_3$. By definition of $A_3$, there exists $b' \in U_2-\{b\}$
 such that $xb'z \in L_{2,3}(x)$. This gives us $abx, abz, xb'z\in \HH$,
forming an $F_5$.

(iii) Suppose for contradiction that $|L_{2,2}(x)|\ge   2\mu n^2$.
By Lemma~\ref{matching} $\{E-x:\ E\in L_{2,2}(x)\}$ contains a
matching $G$ with size at least $\mu n$. Then using
 Definition~\ref{defdensity} (iii)
(with $G, A_i=U_1-x, A_j=U_3$)
   we find $b,b'\in U_2,\  a\in U_1,\  c\in U_3$ such that $abc, ab'c, xbb'\in \HH$,
 forming an $F_5\subset \HH$, a contradiction.

The proof of (iv) is identical to (ii) and of (v) is to (iii).
\end{proof}

\subsection{Getting rid of bad edges - A Progressive Induction}

Here we have to do something similar to  the previous section,
however, as we get rid of only a few edges, the computation needed
is more delicate. We shall do progressive induction on the number of
vertices. The general idea is that we remove some vertices of a bad
edge, and count the number of ways it could
 have been joined to the rest of the hypergraph.

 We shall prove \eqref{induction} via induction on $n$. Fix an $n_0$
such that $1/n_0$ is much smaller than any of our constants, and all
of our prior lemmas and theorems are valid for every $n\ge n_0$. Let $C>10$ be
sufficiently large that \eqref{induction} is true for every $n\le
n_0$.

Let $Forb'(n, F_5, \eta, \mu)$ be the set of hypergraphs $\HH \in
Forb(n,F_5,\eta,\mu)$ having an optimal partition with a bad edge.
Our  final step is to give an  upper bound $|Forb'(n, F_5, \eta,
\mu)|$. There are two types of bad edges, one which is completely
inside of a class, and the one which intersects two classes.

Let the bad edge be $xyz$, and the optimal partition be $U_1,U_2,U_3$.  Without loss of generality assume that
$x,y\in U_1$.

In an $\HH \in Forb'(n,F_5,\eta,\mu)$, $x,y,z$ could be chosen at
most $n^3$ ways, the optimal partition of $\HH$ in at most $3^n$ ways and
the hypergraph $\HH-\{x,y\}$ in at most $|Forb(n-2,F_5)|$ ways. By Lemma \ref{lowdegree} each of $|L_{1,1}(x)|,\ |L_{1,1}(y)|,\
|L_{1,2}(x)|,\ |L_{1,2}(y)|,\ |L_{1,3}(x)|,\ |L_{1,3}(y)|,\
|L_{2,2}(x)|,$ $ |L_{2,2}(y)|,\ |L_{3,3}(x)|,\ |L_{3,3}(y)|$ is at most $2\mu
n^2$, therefore the number of ways the bad edges could be joined to
$x,y$ is at most
$$\left(\sum_{i\le 2\mu n^2}\binom{n^2/2}{i}\right)^{10}\le 2^{10 H(4\mu)n^2}.$$

The key point is that for any $(u,v)\in (U_2-z)\times (U_3-z)$,
we cannot have both $xuv,yuv\in \HH$ otherwise they form with $xyz$ a
copy of $F_5$. Together with Definition \ref{defdensity} part (iv),
we conclude that the number of ways to choose the crossing edges containing $x$ or $y$ is at most
$$3^{|U_2||U_3|}2^{2n}\le 3^{\frac{n^2}{9}+\mu n^2}.$$
Note that the $2^{2n}$ estimates the number of ways having edges containing $u,z$ or $vz$, as for these pairs we do not have any restriction.

Putting this together,
\begin{equation}
|Forb'(n,F_5,\eta,\mu)|\ \le\ n^3 3^n |Forb(n-2,F_5)|\cdot  2^{10
H(4\mu)n^2} 3^{\frac{n^2}{9}+\mu n^2}.\end{equation}
 By the induction hypothesis, this is at most
$$  n^3 3^n (1+2^{C(n-2)-\frac{2(n-2)^2}{45}})T(n-2)
2^{10H(4\mu)n^2}3^{\frac{n^2}{9}+\mu n^2}.$$
Using (\ref{tnn-2})  this is upper bounded by
  $$ n^5 3^n\left(1+2^{C(n-2)-\frac{2(n-2)^2}{45}}\right)
2^{(90H(4\mu)+ \log_2 3+9\mu-2+\frac9n)\frac{n^2}{9}}\cdot T(n).$$
As mentioned before, the crucial point in the expression above is that $\log_2 3-2<0$.
More precisely, since  $n>n_0$,  $\log_2 3<1.59$ and
$90H(4\mu)+9\mu<0.001$, we have
$$\left(90H(4\mu)+ \log_2 3+9\mu-2+\frac9n\right)\frac{n^2}{9}<-\frac{2n^2}{45}.$$
Consequently,
$$ |Forb'(n,F_5,\eta,\mu)|\ \le\ n^5 3^n\left(1+2^{C(n-2)-\frac{2(n-2)^2}{45}}\right)  \cdot  2^{-\frac{2n^2}{45}}T(n)< \frac{1}{10} 2^{Cn-\frac{2n^2}{45}}T(n).  $$

Now we can complete the proof of \eqref{induction} by upper bounding $|Forb(n,F_5)|$ as follows:
$$ |Forb(n,F_5)-Forb(n, F_5, \eta)| + |Forb(n,F_5,\eta)-Forb(n,F_5,\eta,\mu)| + |Forb'(n,F_5,\eta,\mu)|+T(n)$$
$$ \ < \  2^{(1-\nu)\frac{n^3}{27}} + 2^{n^3(\frac{1}{27}-\frac{\mu^3}{40})} +  \frac{1}{10}
2^{Cn-\frac{2n^2}{45}}T(n) + T(n)$$ $$ <
 (1+2^{Cn-\frac{2n^2}{45}})T(n),$$
where the last inequality holds due to $T(n)>2^{s(n)}>2^{\frac{n^3}{27}-O(n^2)}$.
This completes the proof of the theorem. \qed

\end{document}